\documentclass[10pt,oneside,a4paper]{article}
\usepackage{titlesec}
\usepackage{authblk}
\titleformat{\subsection}[runin]{\normalfont\bfseries}{\thesubsection.}{.5em}{}[.~ ]
\titlespacing{\subsection}{0pt}{1.5ex plus .1ex minus .2ex}{0pt}
\usepackage{graphicx}
\usepackage{subcaption}
\usepackage{wrapfig}
\usepackage{amsfonts,amssymb,amsmath,amsthm}
\usepackage{enumerate}
\usepackage{url}
\usepackage[mathcal,mathscr]{euscript}
\usepackage[dvipsnames]{xcolor}
\usepackage{hyperref}
\hypersetup{colorlinks=true,linkcolor=Magenta,citecolor=PineGreen}
\usepackage{float}
\usepackage{tikz-cd}
\usepackage[most]{tcolorbox}
\usepackage{epigraph}

\theoremstyle{plain}
\newtheorem{thm}[equation]{Theorem}

\newtheorem{lemma}[equation]{Lemma}

\newtheorem{prop}[equation]{Proposition}

\theoremstyle{definition}

\newtheorem{defi}[equation]{Definition}

\newtheorem{rmk}[equation]{Remark}
\newcommand{\pp}{{\pmb{p}}}
\newcommand{\qq}{{\pmb{q}}}
\newcommand{\xx}{{\pmb{x}}}
\newcommand{\yy}{{\pmb{y}}}
\newcommand{\zz}{{\pmb{z}}}

\newcommand{\vv}{\pmb{v}}

\newcommand{\re}{\mathrm{\mathrm{Re}\,}}

\newcommand{\im}{\mathrm{\mathrm{Im}\,}}

\newcommand{\PP}{\mathbb{P}}
\newcommand{\HH}{\mathbb{H}}

\newcommand{\CC}{\mathbb{C}}

\newcommand{\KK}{\mathbb{K}}

\newcommand{\SP}{\mathbb{S}}
\newcommand{\RR}{\mathbb{R}}

\title{Complete totally geodesic subsets of the complex hyperbolic plane: an elementary classification}
\author[$\dagger$]{Hugo C.~Bot\'os\footnote{Supported by S\~ao Paulo Research Foundation (FAPESP).}}
\affil[$\dagger$]{\small{Departamento de Matem\'atica Aplicada, IME, Universidade de S\~ao Paulo, S\~ao Paulo, Brasil\authorcr hugocbotos@usp.br}}
\author[$\dagger\dagger$]{Carlos H.~Grossi}
\affil[$\dagger\dagger$]{\small{Departamento de Matem\'atica, ICMC, Universidade de S\~ao Paulo, S\~ao Carlos, Brasil\authorcr grossi@icmc.usp.br}}

\date{}

\begin{document}
	
\maketitle

\begin{center}
	\large\textbf{Abstract}
\end{center}
The non-trivial complete totally geodesic submanifolds of the complex hyperbolic plane $\HH_\CC^2$ are the complex geodesics and the real planes. We present two new proofs for this fact. One is a short proof based on an algebraic formula for the Riemann curvature tensor due to S. Anan$'$in and C. Grossi and resembles the traditional proof using Lie theory. The other is purely elementary and geometric, relying on the structures in $\HH_\CC^2$ instead of general theories. In this second approach, we prove a slightly stronger result: the only non-trivial complete totally geodesic subsets of $\HH_\CC^2$ are the complex geodesics and the real planes without assuming that the subsets are submanifolds a priori. This second proof is also intriguing for only making use of elementary geometric constructions. 

\bigskip

\textit{Dedicated to the cherished memory of Sasha Anan$'$in, from whom we had the honor of learning much about many things. His love for mathematics, unparalleled talent, and infectious enthusiasm, coupled with the profound depth of his scientific vision, left an indelible imprint upon our hearts. Sasha was a singular individual, and since his passing, not a day has gone by without us remembering him with fondness and great affection.}
 
\section{Introduction}

In the introduction to his seminal book [Gol], W.~Goldman states that ``[...] explicit calculations are more valuable than quoting general theorems''. Adopting this philosophy, we present a couple of new proofs to well-known facts regarding the classification of complete totally geodesic submanifolds of the complex hyperbolic plane. To the best of our knowledge, all known proofs of such classification rely on Lie theory. The elementary and purely geometric approach we follow here was strongly influenced by what we learned from Sasha Anan$'$in and his way of doing and thinking about mathematics.

The complex hyperbolic plane $\HH_\CC^2$ can be seen in simple words as the unit open ball in $\CC^2$ with its natural Kähler geometry. 
It is one of the most elementary spaces with negative non-constant curvature, often utilized in various fields such as Kähler geometry, Riemannian geometry, geometric topology, dynamical systems, etc.

This work aims to offer straightforward proof for the classification of complete totally geodesic submanifolds of the complex hyperbolic plane. When we say elementary, we refer to a proof that exclusively employs concepts from linear algebra and projective geometry. Specifically, our classification approach is purely geometric, meaning we rely solely on geometric configurations involving points, geodesics, and bisectors. This approach gives our work an incidence geometry flavor.

More precisely, we define a {\it complete totally geodesic subset} $C$ of the complex hyperbolic plane as a closed subset with the following property: a geodesic $G$ containing two distinct points of $C$ is a subset of $C$ (we take $G$ to be the entire geodesic, not only the segment connecting the points). In Section \ref{sub: classification 2}, we conclude that $C$ must be either $\varnothing$, a point, a geodesic, a complex geodesic, a real plane, or $\HH_\CC^2$. Thus, we don't require that $C$ be a submanifold initially. Complex geodesics and real planes are hyperbolic planes isometrically embedded in $\HH_\CC^2$ as holomorphic and Lagrangian submanifolds, respectively. 

The proof for the classification presented in Goldman's book relies on the idea of the Lie triple system, a tool from the theory of symmetric spaces (see \cite[Section 3.1.11]{goldmanbook}).

We also provide an alternative proof for the classification relying on the algebraic formula for the Riemann curvature tensor present in the work of Sasha Anan$'$in and Carlos Grossi {\it Coordinate-Free Classic Geometries} \cite{coordinatefree}, which follows the philosophy described above. This proof resembles the argument made in Goldman's book but does not rely on general theory.

The importance of understanding these complete totally geodesic submanifolds is enormous. For instance, in the construction of complex hyperbolic $4$-manifolds via tessellation, it is necessary to use $3$-manifolds as faces for the fundamental polyhedron and, since there are no totally geodesic three-dimensional submanifolds in $\HH_\CC^2$, there are issues when constructing them, because these polyhedra are non-convex (see \cite{poincareth}). Another example is in computations involving the Toledo invariant, where the complex geodesics and real planes are relevant in the integration of the symplectic form of $\HH_\CC^2$ over $2$-cells because the former is holomorphic and the latter is Lagrangian (in \cite{tol}, for instance, real planes and complex geodesics are used to prove the Toledo rigidity).

    \section{Real and complex hyperbolic spaces}
     We present here some essential aspects of the real and complex hyperbolic geometries. For a more in-depth discussion, see \cite{goldmanbook}, \cite{coordinatefree}, and \cite{discbundles}.
    
    \subsection{Hermitian and Riemannian geometry of complex hyperbolic spaces}
    
    Consider the field $\KK= \RR$ or $\CC$ and let $V$ be a finite-dimensional $\KK$-linear space endowed with a Hermitian form $\langle \cdot ,\cdot \rangle$ with signature $-+\cdots+$. 

    \medskip
    
    \begin{rmk}  \it The standard Hermitian $\KK$-linear space with signature $-+\cdots+$ is $\KK^{n+1}$ with the Hermitian form $$\langle x,y \rangle : = - x_0\overline{y_0} + \sum\limits_{j=1}^n x_j\overline{y_j}$$
    where $x=(x_0,\ldots,x_n)$ and $y=(y_0,\ldots,y_n)$.
    \end{rmk}
    \medskip
    
    Associated with the Hermitian linear space $V$ we have the ball
    $$B(V) = \{ \pp \in \PP_\KK(V): \langle p,p \rangle<0\}.$$

    Here $\PP_\KK(V)$ is the projective space of $V$ for the field $\KK$ and $p\in V$ is a generic representative of the point $\pp \in \PP_\KK(V)$.
    \medskip

    \begin{rmk} \it In order to see that $B(V)$ is a ball, note that if $n+1=\dim_\KK V$, then we can take an orthonormal basis for $V$ and identify $\KK^{n+1}$ endowed with the previously described Hermitian form. Thus, 
    $$B(V) \simeq B(\KK^{n+1}) = \{[1:x_1:x_2: \cdots: x_n] \in \PP_\KK^n: -1+ |x_1|^2+\cdots+|x_n|^2<0\}, $$
    which is diffeomorphic to the unit ball in $\KK^n$.
    \end{rmk}

    \medskip

    The tangent space at $\pp \in B(V)$ is naturally identified with the space of $\KK$-linear maps $\KK p \to p^\perp$, where $p^\perp$ is the orthogonal complement of $p$ in $V$ concerning $\langle \cdot,\cdot \rangle$.
    
    The ball $B(V)$ has the natural positive-definite Hermitian metric
    $$\langle t_1,t_2 \rangle := - \frac{\langle t_1(p),t_2(p) \rangle}{\langle p,p \rangle}$$
    where $\pp \in B(V)$ and $t_1,t_2 \in T_\pp B(V)$. This space is the hyperbolic space of the Hermitian linear space $V$.
    
    For $V = \KK^{n+1}$ and $\langle x,y \rangle : = - x_0\overline{y_0} + \sum\limits_{j=1}^n x_j\overline{y_j}$, we denote the corresponding hyperbolic space by $\HH_\KK^n$. The space $\HH_\RR^n$ is the {\bf $n$-dimensional real hyperbolic space} and $\HH_\CC^n$ is the {\bf $n$-dimensional complex hyperbolic space}. Observe that $V$ always admits an orthonormal basis and therefore $B(V)$ is isometric to $\HH_\KK^n$, where $n+1=\dim_\KK V$.
    
    The complex hyperbolic space is a Kähler manifold and as such it has a Riemannian metric $g =\re \langle \cdot,\cdot \rangle$ and a symplectic form $\omega = \im \langle \cdot,\cdot \rangle$.
    
    Following \cite[4.5 Curvature tensor]{coordinatefree}, the Riemann curvature tensor is given by the algebraic expression
    \begin{equation}\label{riemann curvature tensor}
        R(t_1,t_2)s = \langle t_2,t_1 \rangle s + \langle s,t_1 \rangle t_2 - \langle t_1,t_2 \rangle s - \langle s,t_2 \rangle t_1,
    \end{equation}
    when using the convention $R(X,Y)Z := \nabla_X \nabla_Y Z -\nabla_Y \nabla_X Z - \nabla_{[X,Y]} Z$.
    
    \subsection{Complex geodesics, real planes, and bisectors in $\HH_\CC^2$}
   
    A {\bf geodesic} in $\HH_\CC^2$ is a set $G=\HH_\CC^2 \cap \PP_\CC(W)$, where $W$ is a real two-dimensional subspace $\CC^3$ such that $\langle \cdot,\cdot \rangle|_{W \times W}$ is real-valued with signature $-+$. Two distinct points $\pp,\qq \in \HH_\CC^2$ determine a unique geodesic, provided by $W= \RR p \oplus \langle p,q \rangle \RR q$. For $\pp \in \HH_\CC^2$ and $t \in T_\pp G$ with $\langle t,t \rangle = 1$, the curve $\theta \mapsto \cosh(\theta) p + \sinh(\theta)t(p)$ parametrizes the geodesic with velocity one. 
    
    The vertices of the geodesic $G$ are $\vv_1,\vv_2 \in \partial \HH_\CC^2$, obtained by intersecting $\PP_\CC(W)$ with $\partial \HH_\CC^2$. In particular, $\langle v_1,v_1 \rangle =\langle v_2,v_2 \rangle=0$ and $\langle v_1,v_2 \rangle \neq 0$. Furthermore, if $v_1,v_2 \in W$, then $\langle v_1,v_2 \rangle$ is real. Up to rescaling $v_1,v_2$, we may assume $\langle v_1,v_2 \rangle = 1/2$ and obtain the parametrization $\alpha \mapsto \alpha v_1 - \alpha^{-1}v_2$, with $\alpha >1$, for $G$ in terms of its vertices. The definition for geodesics is analogous to $\HH_\KK^n$.
    
    The hyperbolic space of real dimension $2$ appears in two forms in $\HH_\CC^2$: as complex geodesics (complex submanifolds) and as real planes (Lagrangian submanifolds). 
    
    A {\bf complex geodesic} $L$ is a non-trivial intersection of a complex projective line with $\HH_\CC^2$. More precisely, $L = \HH_\CC^2 \cap \PP_\CC(W)$, where $W$ is an two-dimensional complex subspace of $\CC^3$ with signature $-+$. Complex geodesics are isometric to $\HH_\CC^1$, the Poincaré disc, which has $-4$ curvature.
    
    A {\bf real plane} is a set $R=\HH_\CC^2 \cap \PP_\CC(W)$, where $W$ is a real three-dimensional subspace $\CC^3$ such that $\langle \cdot,\cdot \rangle|_{W \times W}$ is real-valued and has signature $-++$. The space $R$ is isometric to $\HH_\RR^2$, the Beltrami-Klein plane, and has curvature $-1$.
    
    Our objective is to prove that these two isometrically embedded hyperbolic planes are the only non-obvious totally geodesics subsets of $\HH_\CC^2$ from a purely geometric viewpoint.
    
    Now we examine bisectors. Geometrically, a bisector is the locus of points equidistant to two fixed distinct points in $\HH_\CC^2$. Nevertheless, we will not use this definition. Instead, we define bisectors as follows: given a geodesic $G$, provided by a real two-dimensional subspace $W \subset \CC^3$, the {\bf bisector} $B$ determined by $G$ is the set $\HH_\CC^2 \cap \PP(W \oplus \CC u)$, where $u$ is a unit vector perpendicular to the complex two-dimensional subspace $W+i W \subset \CC^3$. We can view $B$ as a disc bundle over the curve~$G$:
    $$B = \bigsqcup_{\xx \in G} \HH_\CC^2 \cap \PP(\CC x+ \CC u).$$
    Note that each $\HH_\CC^2 \cap \PP(\CC x+ \CC u)$ is a complex geodesic. These complex geodesics are called {\bf slices} of the bisector and the geodesic $G$ is called the {\bf real spine} of the bisector.
    The complex geodesic $\HH_\CC^2 \cap \PP(W+iW)$ is called {\bf complex spine} and is orthogonal to each slice. The bisector $B$ can also be seen as well as the union of all real planes passing through $G$:
    $$B = \bigcup_{z \in \SP^1} \HH_\CC^2 \cap \PP(W+ \RR z u),$$
    where $\SP^1$ is the unit circle in $\CC$. These real planes only have $G$ in common and are called {\bf meridians}.
    
    As a real algebraic hypersurface, the bisector $B$ is described by the equation
    $$\im\left(\frac{\langle x,g_1 \rangle\langle g_2,x \rangle}{\langle g_1,g_2 \rangle}\right)=0,$$
    where $\pmb{g}_1,\pmb{g}_2$ are distinct points of $\overline{G} = G \cup \partial G$, (see \cite[Proposition A.11]{discbundles}). For a detailed discussion about geodesics, complex geodesics, real planes, and bisectors, see \cite[Appendix A]{discbundles}.
    \begin{rmk} \it
    From this equation, it is easy to see that $B$ is not totally geodesic.
    Indeed, consider the vertices $\vv_1,\vv_2$ of $G$ and take representatives satisfying $\langle v_1,v_2 \rangle = 1/2$. Consider two points $\xx,\yy \in B \setminus G$ in orthogonal meridians and distinct slices. Up to careful choice of $v_1,v_2,u$, we can represent $\xx$ by $x = v_1-v_2 + r u$ and $\yy$ by $y = \alpha v_1 - \alpha^{-1}v_2+i s u$, where $\alpha$ is positive and not $1$, and $r,s \in (-1,1)$  are non-zero. The point $q :=x-\frac{\langle x,y \rangle}{|\langle x,y \rangle|} y$ represents a point of the geodesic connecting $\xx$ to $\yy$ because $q \in W$ for $W=\RR x + \langle x,y \rangle \RR y $  and
    $$\langle q,q \rangle  =-2 + r^2 + s^2 - \frac{\sqrt{\left(\alpha ^2+1\right)^2+4 \alpha ^2 r^2 s^2}}{\alpha}<0.$$  
    From
     $$\im\left(\frac{\langle q,v_1 \rangle\langle v_2,q \rangle}{\langle v_1,v_2 \rangle}\right)=\frac{r s (\alpha^2-1)}{\sqrt{\left(\alpha ^2+1\right)^2+4 \alpha ^2 r^2 s^2}}\neq 0$$
    we conclude that the point corresponding to $\qq$ in $G = \PP_\CC(W) \cap \HH_\CC^2$ does not belong to the bisector $B$. Therefore, $B$ is not totally geodesic.
    \end{rmk}

    \section{Complete totally geodesic submanifolds of $\HH_\CC^2$}
    We provide the classification of the complete totally geodesic submanifolds of $\HH_\CC^2$ via classic geometries, in the sense of \cite{coordinatefree}. Two proofs of the classification will be presented. The first argument is inspired by the theory of symmetric spaces and it is a well-known approach with the difference being that we make use of the Riemann tensor curvature in the language of classic geometries, as stated in \eqref{riemann curvature tensor}. The second argument is purely geometric and elementary. It does not make use of Riemannian geometry, only general properties of geodesics and bisectors.
    \subsection{Classification via the Riemann curvature tensor}

    Let $\HH$ be a complete Riemannian manifold and let $S$ be a totally geodesic submanifold of $\HH$. By totally geodesic we mean that the geodesics of $S$ are geodesics of $\HH$. Equivalently, it means that the Levi-Civita connection of $S \subset \HH$ is the restriction of the Levi-Civita connection of $\HH$ to~$S$. 
    
    In particular, the curvature tensor of $S$ is the restriction of the curvature tensor of the ambient space $\HH$. Therefore, if $R$ is the curvature tensor of $\HH$, then for every $\pp \in S$ we must have $R(t_1,t_2)t_3 \in T_\pp S$ for all $t_1,t_2,t_3 \in T_\pp S$. We use this property to classify the complete totally geodesic submanifolds of $\HH_\CC^2$. This strategy is widely known and it is at the core of the classification of totally geodesic submanifolds of symmetric space through a technique called Lie triple system, see \cite[Subsection 3.1.11]{goldmanbook} for the $\HH_\CC^n$ case and \cite[Chapter XI, Section 4]{kobayashi} for general symmetric spaces.

    Now we classify the complete totally geodesic connected submanifolds $S$ of the complex hyperbolic plane $\HH_\CC^2$. Geodesics are the only one-dimensional instances, and $\HH_\CC^2$ is the only four-dimensional one. So, we just need to investigate the cases where $S$ has dimensions $2$ and $3$.
    
    \begin{prop}
    A complete totally geodesic connected surface $S$ embedded in $\HH_\CC^2$ is a complex geodesic or a real plane. 
    \end{prop}
    \begin{proof}
    Let $\pp$ be a point of $S$ and let $t_1,t_2$ be an orthonormal basis to $T_\pp S$ concerning the Riemannian metric $g=\re \langle \cdot,\cdot \rangle$. If the vectors $t:=t_1$ and $n$ form an orthonormal basis for $T_p \HH_\CC^2$ concerning the Hermitian metric $\langle \cdot,\cdot \rangle$, then we can write $t_2$ as a linear combination with complex coefficients of $t$ and $n$. Since $\re \langle t_1,t_2 \rangle = 0$ we can write
    $t_2 = ia t + b n$ where $a$ is real and $a^2 + |b|^2 = 1$.
    
    Since $S$ is totally geodesic, the vector 
    $$R(t_1,t_2)t_2 = -(1+3a^2) t + 3 i a b n$$
    is tangent to $S$ at $\pp$. Thus, we conclude that $i 3ab n \in \RR t_2$. Therefore, $a=0$ or $b=0$.
    If $a=0$, then $\langle t_1,t_2 \rangle= 0$ and $S$ is the real plane determined by $\RR p \oplus \RR t_1(p) \oplus \RR t_2(p)$.
    
    If $b=0$, then $T_p S = \CC t_1$ and $S$ is the complex geodesic determined by $\CC p \oplus \CC t_1(p)$.
    \end{proof}
    
    \begin{prop}
The complex hyperbolic plane  $\HH_\CC^2$ admits no complete totally geodesic three-dimensional submanifold.
    \end{prop}
    \begin{proof} Assume that there exists a three-dimensional complete totally geodesic submanifold $S$ of $\HH_\CC^2$.
    Let $\pp$ be a point of $S$ and let $t_1$ be a unit vector at $T_\pp S$. The the orthogonal complement $t_1^\perp =\{t \in T_\pp \HH_\CC^2: \langle t,t_1 \rangle =0\}$ has real dimension $2$ and therefore intersects $T_\pp S$ non-trivially. Take $t_2$ to be a unit vector in their intersection. Now, adding a unit vector $t_3 \in T_\pp S$ orthogonal to $t_1,t_2$ for the Riemannian metric $g:=\re \langle \cdot, \cdot \rangle$, we obtain an orthonormal basis for the tangent space of $S$ at $p$ to the metric $g$.
    
    Since $t_1,t_2$ is a basis to $T_\pp\HH_\CC^2$, $t_3$ can be written as a linear combination of $t_1,t_2$ with complex coefficients. Additionally, since $g(t_1,t_3) = g(t_2,t_3)=0$, we have $t_3 = ia t_1+ ib t_2$, with $a,b \in \RR$. 
    
    The vectors 
        \begin{align*}
    R(t_1,t_2)t_3 &=  -i b t_1 + i a t_2 \\
    R(t_1,t_3)t_1 &= 4 i a t_1 + i b t_2
    \end{align*}
    must be a linear combination of $t_1,t_2,t_3$ with real coefficients because $S$ is totally geodesic. Thus, we conclude that $a=b=0$, implying that the vector $t_3$ is null, a contradiction. 
    \end{proof}
    \subsection{Classification via basic linear algebra and projective geometry}
    \label{sub: classification 2}
    \begin{defi} A subset $C$ of $\HH_\CC^2$ is said to be complete totally geodesic when it satisfies the following properties:
    \begin{itemize}
        \item $C$ is closed;
        \item if $\pp,\qq \in C$ are two distinct points then the geodesic defined by these points is contained in~$C$. 
    \end{itemize}
    \end{defi}
    
    Observe that in the second item of the above definition, we are assuming that the whole geodesic is contained in the $C$, not just the segment connecting the two points.
    
    \begin{lemma} \label{lemma complex geodesic and R plane}
    Let $C$ be a complete totally geodesic subset of $\HH_\CC^2$ containing a point $\pp$ and a geodesic $G$ not passing through this point. 
    
    \begin{itemize}
        \item If $G$ and $p$ belong to a complex geodesic $L$, then $L \subset C$;
        \item If $G$ and $\pp$ belong to a real plane $R$, then $R \subset C$.
    \end{itemize} 
    \end{lemma}
    \begin{proof}
    Let $D$ be a complex geodesic or a real plane containing $G$ and $\pp$. The geodesic $G$ divides the disc $D$ into two open components $D_1$ and $D_2$, where we assume $\pp \in D_1$. Consider a geodesic $G'$ passing through $\pp$ and intersecting $G$. In particular $G' \subset C$ and we can consider $\pp' \in G' \cap D_2$. For $\xx \in D_1$, we consider the geodesic passing through $\xx$ and $\pp'$ and this geodesic will be contained in $C$ because it intersects $G \subset C$ and contains $\pp' \in C$. Similarly, for $\xx \in D_2$, the geodesic passing through $\xx$ and $\pp$ crosses $G$ and, as consequence, it is contained in $C$. Therefore, $D \subset C$.
    \end{proof}

    \begin{prop}\label{proposition point and real plane}
    If a complete totally geodesic subset $C$ of $\HH_\CC^2$ contains a point $\pp$ and a real plane $R$ not passing through $\pp$, then $C = \HH_\CC^2$.
    \end{prop}
    \begin{proof} Consider the complex projective line $L=\PP(\pp^\perp)$. Since $\pp \not \in R$, $L \cap R$ is a positive point, which we denote by $\qq_3$. Indeed, if $W$ is a $3$-dimensional $\RR$-linear subspace of $V$ defining the real plane $R$, then there exists an orthonormal basis $e_1,e_2,e_3$ for $V$ concerning the Hermitian form, with $e_1$ negative and $e_2,e_3$ positive, such that $W = \RR e_1 \oplus \RR e_2 \oplus \RR e_3$. Since $\pp \not \in R$ we have a representative $p=e_1 + \beta e_2 + \gamma e_3$ with $\im \beta \neq 0$ or $\im \gamma \neq 0$. By taking $q_3 = xe_1+ye_2+ze_3$ with $x,y,z \in \RR$ we can solve the equation $\langle q_3,p \rangle = -x + y\beta +z \gamma =0$ and find the coefficients for $q_3$.
    
     We now prove that there exists a geodesic $G$ in $R$ such that the bisector $B$ defined by $G$ contains~$\pp$.
     
     The $\RR$-linear space $W$ has an orthonormal basis $q_1,q_2,q_3$, where $q_1$ is negative and $q_3$ is a representative for the point $\qq_3$. Consider the geodesic $G$ in $R$ connecting $\qq_1$ and $\qq_2$, that is, $G$ is obtained from $\RR q_1 \oplus \RR q_3$. Since $\pp$ admits a representative $p=q_1+k q_2$ with $k\not \in \RR$, the point $\pp$ belongs to the bisector $B$ generated by the geodesic $G$. In particular, the real plane $R$ contains the geodesic $G$ and therefore it is a meridian of the bisector $B$ determined by such geodesic.
    
    Let $R'$ be the meridian of $B$ containing $p$, that is, this real plane is provided by $W' = \RR q_1 \oplus \RR k q_2 \oplus \RR q_3$. Observe that the real plane $R'$ is contained in $C$. Since, $R$ and $R'$ are different $\RR$-planes composing the bisector $B$ and $R, R' \subset C$, then by the lemma \ref{lemma complex geodesic and R plane}, the bisector $B$ is contained in $C$ (just consider the intersections of the slices of the bisector $B$ with $R$ and $R'$).
    
    Since $B$ is not totally geodesic, there exists $\pp' \in C \setminus B$. Furthermore, since $B$ is a hypersurface dividing $\HH_\CC^2$ into two regions, we conclude that $C=\HH_\CC^2$.
    \end{proof}

    \begin{lemma}
    If $C$ is a complete totally geodesic subset of $\HH_\CC^2$ containing the geodesic $G$ and a point $\pp$, with $\pp \not \in G$, then the geodesics connecting the point $\pp$ to the vertices of the geodesic $G$ are contained in $C$ as well.
    \end{lemma}
    \begin{proof}
    Consider a Riemannian sphere $S$ of $\HH_\CC^2$ centered at $\pp$ with a small radius (Figure \ref{vertex}). 
    
    For each $\xx \in G$, the geodesic segment connecting $\pp$ to $\xx$ cuts the sphere $S$ at a point $\qq_{\xx}$. Thus the map $G \to S$ given by $\xx \mapsto \qq_{\xx}$ draws a curve $\gamma$ over the sphere $S$. Since each geodesic connecting $\pp$ to $\xx \in G$ is contained in $C$, the curve $\gamma$ is contained in $C$ as well. Due to the compactness of the sphere $S$, the curve $\gamma$ has two extreme points, which belong to $C$ because $C$ is closed. If $\qq$ is one of the extreme points, the geodesic determined by $\qq$ and $\pp$ is contained in $C$ and it reaches one of the vertices of $G$, thus proving the result.
    \end{proof}
    \begin{figure}[H]
	\centering
	\includegraphics[scale = .75]{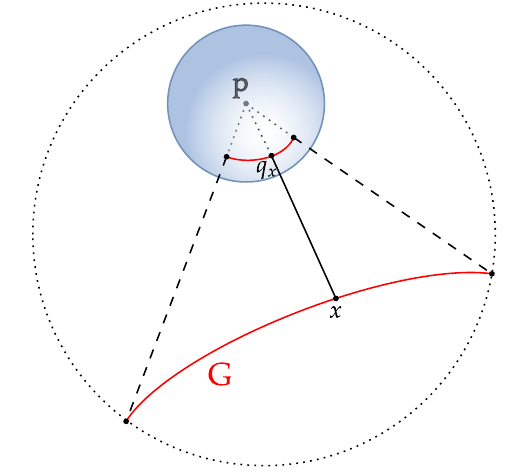}
	\caption{Projecting the geodesic over the sphere.}
    \label{vertex}
    \end{figure}

    \begin{prop}
    Consider a geodesic $G$ and a point $\pp$ contained in a complete totally geodesic subset $C$. If $\pp$ and $G$ do not share a common complex geodesic or a common real plane, then $C=\HH_\CC^2$.
    
    \end{prop}
    \begin{proof} Let $\vv_1,\vv_2$ the vertices of the geodesic $G$. Projecting the point $\pp$ orthogonally on $G$ we obtain a point $\pp'$. Take representatives form $\vv_1,\vv_2,\pp'$ such that  $\langle v_1,v_2 \rangle = 1/2$ and $p' = v_1-v_2$.

    \begin{figure}[H]
	\centering
	\includegraphics[scale = .75]{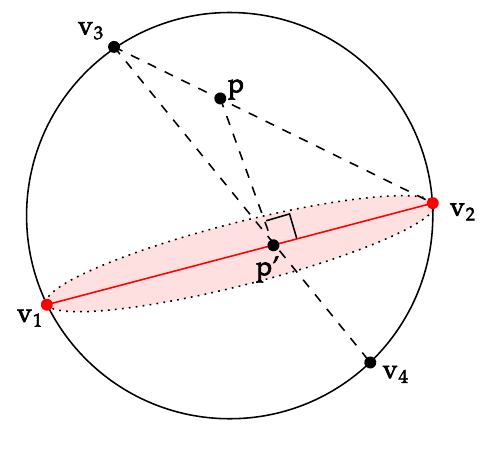}
	\caption{Configuration of the points $\pp,\pp',\vv_1,\vv_2,\vv_3,\vv_4$.}
    \label{general case}
    \end{figure}
    
    Consider the geodesic determined by $\vv_2$ and $\pp$ and denote by $\vv_3$ the other vertex of such geodesic. Additionally, this geodesic is contained in $C$, because of $G\subset C$ and $p\in C$. Let $u$ be a vector perpendicular to $\CC v_1 + \CC v_2$. Since $v_1,v_2$ determines a complex geodesic, $u$ is positive and we may assume $\langle u,u \rangle = 1$. The point $\vv_3$ can be represented by $v_3 = \varepsilon v_1- \overline\varepsilon v_2+r u$, where $\varepsilon\in \CC$ and $r\geq 0$, by modifying the representatives of $v_1,v_2,u$. Observe that $v_3$ and $G$ do not share a common complex geodesic and, therefore, $r>0$. Since $\langle v_3,v_3 \rangle = 0$, we have $\re(\varepsilon^2) = r^2>0$. Thus, $\varepsilon$ is not purely imaginary. On the other hand, $\varepsilon$ is not real as well because $G$ and $\vv_3$ do not lay down in a common real plane. We will represent $\pp$ by $p = \alpha v_3 - \alpha^{-1} \varepsilon v_2 \in \RR v_3 + \langle v_3,v_2 \rangle \RR v_2$ for some $\alpha>0$.

    The geodesic determined by $\vv_3$ and $\pp'$ has a point $\vv_4$ as the other vertex. 
    Note that such geodesic is contained in $C$ because the geodesic connecting $\vv_2$ and $\vv_3$ is in $C$ and $\pp' \in C$. Additionally, $v_4 = 2\re(\varepsilon) p' - v_3$ is a representative for $\vv_4$.

    Now we show that the geodesic $G'$ connecting $\pp,\vv_4$ cuts the bisector $B$ determined by the geodesic $G$ connecting $\vv_1$, $\vv_2$ at a point $\xx$. Furthermore, we prove that $\xx$ does not belong to $G$, thus implying that there exists one real plane $R$ containing $G$ and contained in $C$ (see Figure \ref{bisector two sides}). Applying the Proposition \ref{proposition point and real plane} to the real plane $R$ and the point $\pp$, we conclude that $C = \HH_\CC^2$.   
    
    \begin{figure}[H]
	\centering
	\includegraphics[scale = .75]{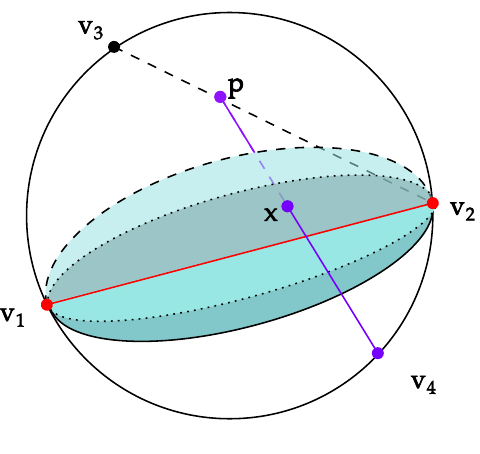}
	\caption{Intersecting the geodesic connecting $\pp$ and $\vv_4$ with the bisector of $G$.}
    \label{bisector two sides}
    \end{figure}

    The bisector $B$ determined by the geodesic $G$ is given by the algebraic equation 
    $$\im \langle v_1,z \rangle \langle z, v_2 \rangle = 0, \quad \zz \in \HH_\CC^2.$$
    A hypersurface separates $\HH_\CC^2$ into two connected components. Since
    $$\im \langle v_1,p \rangle \langle p, v_2 \rangle = -\frac14 \alpha^2 \im(\varepsilon^2),$$
    and
    $$\im \langle v_1,v_4 \rangle \langle v_4, v_2 \rangle = \frac14  \im(\varepsilon^2)$$
    have values with opposite signs,
    the points $\pp,\vv_4$ are in opposite components.
    
    Therefore, there is a point $\xx$ in the geodesic $G'$ connecting $p,v_4$ and in the bisector $B$. Now we prove that $\xx$ does not belong to the geodesic $G$. So, assume that $\xx \in G$, and let us reach a contradiction. Since $\xx$ is in the complex geodesic connecting $\pp,\vv_4$, we can take $x = c v_4+p$ with $c \in \CC$.
    Thus,
    $$x = c(2\re(\varepsilon) v_1 - 2 \re(\varepsilon) v_2 -  v_3) +  \alpha v_3- \alpha^{-1} \varepsilon v_2$$
    
    On the other hand, we are assuming that $\xx \in G$ and, as a consequence, $c=\alpha$ because $x$ can not have a $v_3$ component. Thus, $x = \alpha v_4+p$. Here we have a contradiction, because $\langle p,v_4 \rangle \not \in \RR$, and thus $\alpha v_4+p$ does not represent a point in the geodesic $G'$. Indeed,
   $$\langle p,v_4 \rangle = -\alpha(\re(\varepsilon)^2+|\varepsilon|^2) - \frac{\varepsilon^2}{2\alpha} \not \in \RR.$$
    
    Therefore, $\xx \not \in G$ and the real plane $R$ in the bisector $B$ are determined by $\xx$, and $G$ is contained in $C$. Since $\pp \not \in R$, we conclude that $C= \HH_\CC^2$ via Proposition \ref{proposition point and real plane}.
    \end{proof}
    
    From the previous results, we conclude the following.
    \begin{thm} The only non-trivial complete totally geodesic subsets of $\HH_\CC^2$ are complex geodesics and real planes.
    \end{thm}

\end{document}